\documentclass[preprint,12pt,3p]{elsarticle}
\usepackage{amsmath,amsthm,amsfonts,amssymb}
\usepackage{fullpage}
\usepackage{subfig}
\usepackage{blkarray}
\usepackage{caption}
\usepackage{float}
\usepackage{algpseudocode}
\usepackage{tikz}
\usepackage{hyperref}
\numberwithin{equation}{section}
\usepackage{xcolor}
\usepackage{colortbl}
\usepackage[normalem]{ulem}
\usepackage{sectsty}

\definecolor{astral}{RGB}{46,116,181}
\subsectionfont{\color{astral}}
\sectionfont{\color{astral}}
\usepackage{colortbl}

\DeclareMathAlphabet{\mathpzc}{OT1}{pzc}{m}{it}
\DeclareFontFamily{OT1}{pzc}{}
\DeclareFontShape{OT1}{pzc}{m}{it}{<-> s * [0.900] pzcmi7t}{}
\DeclareMathAlphabet{\mathpzc}{OT1}{pzc}{m}{it}

\usepackage{amsbsy}
\usepackage{amsmath}
\usepackage{accents}
\newlength{\dhatheight}

\newenvironment{frcseries}{\fontfamily{frc}\selectfont}{}

\DeclareMathAlphabet\mathbfcal{OMS}{cmsy}{b}{n}
\definecolor{darkslategray}{rgb}{0.18, 0.31, 0.31}
\definecolor{warmblack}{rgb}{0.0, 0.26, 0.26}

\usepackage{multirow}
\usepackage{mathtools}
\usepackage{algorithm}
\usepackage{algorithmicx}
\makeatletter
\def\BState{\State\hskip-\ALG@thistlm}
\makeatother
\newtheorem{theorem}{Theorem}[section]
\newtheorem{lemma}[theorem]{Lemma}
\newtheorem{corollary}[theorem]{Corollary}
\theoremstyle{definition}
\usepackage{hyperref}
\newtheorem{definition}{Definition}[section]
\newtheorem{remark}{Remark}[section]

\newcommand{\R}{{\mathbb R}}

\begin{document}

\begin{frontmatter}

\title{Convergence of two-stage iterative scheme for $K$-weak regular splittings of type II with application to  Covid-19 pandemic model}

\vspace{-.4cm}

\author{Vaibhav Shekhar$^b$, 
Nachiketa Mishra$^a$, and Debasisha Mishra$^{*b}$}

\address{$^a$Department of Mathematics,\\
                        Indian Institute of Information Technology, Design and Manufacturing, Kanchipuram  
                        \\email$^a$:  nmishra@iiitdm.ac.in 
                        \\
                          $^b$Department of Mathematics,\\
                        National Institute of Technology Raipur, India.
                        \\email$^*$: dmishra\symbol{'100}nitrr.ac.in. \\}
\vspace{-2cm}

\begin{abstract}
    Monotone matrices play a key role in the convergence theory of regular splittings and different types of weak regular splittings. If monotonicity fails, then it is difficult to guarantee the convergence of the above-mentioned classes of matrices.
In such a case, $K$-monotonicity is sufficient for the convergence of $K$-regular and $K$-weak regular splittings, where $K$ is a proper cone in $\mathbb{R}^n$.
 However, the convergence theory of a two-stage iteration scheme in general proper cone setting is a gap in the literature. Especially, the same study for weak regular splittings of type II (even if in standard proper cone setting, i.e.,  $K=\mathbb{R}^n_+$), is open.
To this end,
we propose convergence theory of two-stage iterative scheme for $K$-weak regular splittings of both types in the proper cone setting. We provide some sufficient conditions which guarantee that the induced splitting from a two-stage iterative scheme is a $K$-regular splitting and then establish some comparison theorems. We also study $K$-monotone convergence theory of the stationary two-stage iterative method in case of a $K$-weak regular splitting of type II. The most interesting and important part of this work is on $M$-matrices appearing in the Covid-19 pandemic model. Finally, numerical computations are performed using the proposed technique to compute the next generation matrix involved in the pandemic model.
\end{abstract}

\begin{keyword}
Linear system;  Two-stage iteration; Convergence; $K$-nonnegativity; $K$-monotonicity; $K$-weak regular splittings; Pandemic model; $M$-matrix; Next generation matrix.
\end{keyword}

\end{frontmatter}

\newpage
\section{Introduction}

Let us consider a real large and sparse $n\times n$ non-singular linear systems of the form
\begin{equation}\label{eqn1.1}
    Ax=b
\end{equation}
that appear in the process of discretization of elliptic partial differential equations  (see \cite{smith, wach}). Whereas matrix equations of the form $AX=B$ appear in the problem of computing the spectral radius of the next generation matrix in a Covid-19 Pandemic model which is shown in Section \ref{covid}.
Iterative methods using matrix splittings is one of the simple methods for finding the iterative solution of \eqref{eqn1.1} and the same idea can also be used to solve $AX=B$. So, our focus is on developing convergence theory of a particular iteration scheme to solve \eqref{eqn1.1}.  A matrix splitting $A=U-V$ leads to the iteration scheme \begin{equation}\label{eqn1.2}
    x_{k+1} = Hx_{k}+c.
\end{equation} 
The matrix $H=U^{-1}V$ in the above equation is called as the {\it iteration matrix}. It is well-known that the iteration scheme \eqref{eqn1.2} converges for any initial vector $x^0$  (or is  called \textit{convergent}) if $\rho(H)<1$, where $\rho(H)$ denotes the spectral radius of the matrix $H$, i.e., maximum of moduli of eigenvalues of $H$.  

A splitting $A=U-V$ is called {\it a $K$-regular splitting} \cite{climent:1999, climent:1998} if $U^{-1}$ exists, $U^{-1}\geq_K 0$ and $V\geq_K 0$ ($A\geq_K 0$ means $AK\subseteq K$ where $K$ is a proper cone, see the next section for more details). A splitting $A=U-V$ is called {\it a $K$-weak regular splitting of type I (or type II)} \cite{climent:1999, climent:1998} if $U^{-1}$ exists, $U^{-1}\geq_K 0$ and $U^{-1}V\geq_K 0$ (or $VU^{-1}\geq_K 0$).  The following convergence theorem was established in \cite{climent:1999}.
\begin{theorem}\textnormal{(Theorem 2.2, \cite{climent:1999})}\label{2.5}\\
  Let $A = U-V$ be a $K$-weak regular splitting of type I (or type II). Then, $A^{-1}$ exists and $A^{-1}\geq_K 0$ if and only if $\rho(U^{-1}V)<1$.
\end{theorem}
\noindent  When $K=\mathbb{R}^{n}_{+}$, then the definitions of $K$-regular and $K$-weak regular type I (or type II) splittings coincide with the definition of regular  \cite{var} and weak regular type I (or type II) \cite{ortega:1967, woz:1994},  respectively. We refer the reader to the book \cite{var} for the convergence results for these classes of matrices.
However, the classical iterative methods are computationally expensive,  which attracts the researcher to develop fast iterative solvers.  In this context,  several comparison results are obtained in the literature (see \cite{climent:1999, climent2:1999, marek:1990} and the references cited therein). 

In 1973, Nichols \cite{nichols} proposed the notion of two-stage iterative method which is recalled next.
Let us consider an iterative scheme of the form
\begin{equation}\label{outer}
Ux_{k+1}=Vx_{k}+b,
\end{equation}
where we take the splitting $A=U-V$ of $A$ for solving a  linear system of the form \eqref{eqn1.1}. The scheme \eqref{outer} is called {\it outer iteration}.
At each step of \eqref{outer}, we must solve the {\it inner equations}
\begin{equation}\label{inner}
Uy=\tilde{b}, \quad \tilde{b}=Vx_{k}+b.
\end{equation}
In the two-stage iterative technique, we solve (\ref{inner}) by another iterative scheme (called {\it inner iteration}) which is formed by using a splitting $U=F-G$, and the same scheme performs $s(k)$  inner iterations. In particular, Frommer and Szyld \cite{szyld} considered the two-stage iterative scheme of the form
\begin{equation}\label{tsiter}
x_{k+1}=(F^{-1}G)^{s(k)}x_{k}+\displaystyle\sum_{j=0}^{s(k)-1}(F^{-1}G)^{j}F^{-1}(Vx_{k}+b), \quad k=1,2,\ldots.
\end{equation}
 For a given initial vector $x_{0}$, the two-stage iterative scheme (\ref{tsiter}) produces the sequence of vectors
\begin{equation}\label{1.6}
x_{k+1}=T_{s(k)}x_{k}+P^{-1}_{s(k)}b , \quad k=0,1,2,\ldots,
\end{equation}
where 
\begin{equation}\label{itmat}
T_{s(k)}=(F^{-1}G)^{s(k)}+\displaystyle\sum_{j=0}^{s(k)-1}(F^{-1}G)^{j}F^{-1}V \quad \text{and} \quad ~P_{s(k)}^{-1}=\displaystyle\sum_{j=0}^{s(k)-1}(F^{-1}G)^{j}F^{-1}.
\end{equation}
We say that the iterative scheme \eqref{tsiter} is {\it stationary} when $s(k)=s$ for all $k,$ while it is {\it non-stationary} when $s(k)$ changes with $k$. The authors of \cite{szyld} established the convergence of \eqref{tsiter} by considering $A=U-V$ as a convergent regular splitting and $U=F-G$ as a convergent weak regular splitting of type I for both stationary and non-stationary scheme \eqref{tsiter}. We refer to \cite{cao} for the stability and error analysis of the two-stage iterative method.
One can also refer a few more convergence and comparison results by Bai and Wang \cite{bai}, but the scope is limited up to type I only. In this article, we aim to establish the convergence of the two-stage iterative scheme \eqref{tsiter} when $A=U-V$ has a $K$-regular splitting and the inner iteration matrix splitting $U=F-G$ has a $K$-weak regular splitting of type II thus expanding the convergence theory of two-stage iterative method.\\

The structure of this paper is as follows. In Section \ref{prlem},  we introduce
some notations and definitions which help to prove the main results. The convergence result is established in Section \ref{MR}. We further analyze the $K$-regularity of the induced splitting from the two-stage iterative scheme and derive a few comparison theorems. In Section \ref{monotone}, we set up the $K$-monotone convergence theorem for the two-stage iterative scheme. Section \ref{covid} deals with a Covid-19 pandemic model and the computation of the next generation matrix involved in this model.

\section{Preliminaries}\label{prlem}
In this section, we collect some basic results required to prove our main results. 
We begin with the notation ${\R}^{n \times n}$ which represents the set of all real matrices of order $n \times n$. 
   We denote the transpose of
   $A\in {\R}^{n \times n}$ by
$A^{T}$. Throughout the paper, all our matrices are real   $n\times n$ matrices unless otherwise stated.
 $\sigma(A)$ denotes the set of all eigenvalues of $A$. By a {\it convergent matrix} $A$, we mean $\displaystyle \lim_{k\to \infty}A^{k}=0$. A matrix $A$ is  convergent if and only if $\rho(A)<1$. We write $K$ and int($K$) to denote a proper cone and the interior of $K$ in $\mathbb{R}^{n}$, respectively. A nonempty subset $K$ of $\mathbb{R}^{n}$ is called a {\it cone} if $0\leq \lambda$ implies $\lambda K\subseteq K$. A cone $K$ is {\it closed} if and only if it coincides with its closure. A cone is a  {\it convex cone} if $K+K\subseteq K$, a {\it pointed cone} if $K\cap (-K)=\{0\}$ and a {\it solid cone} if int$(K)\neq \phi$. A closed, pointed, solid convex cone is called a {\it proper cone}.
 A proper cone induces a partial order in $\mathbb{R}^{n}$ via $x\geq_K y$ if and only if $x-y\geq_K 0$ (see \cite{bpbook} for more details).
 $\pi(K)$ denotes the set of all matrices in $\mathbb{R}^{n\times n}$ which leave a proper cone $K\subseteq \mathbb{R}^{n}$ invariant (i.e., $AK\subseteq K$).
 We now move to the notion of $K$-nonnegativity of a matrix which generalizes the usual nonnegativity (i.e., entry-wise nonnegativity).  $A\geq_K 0$ is equivalent to $A\in \pi(K)$. For $A,B\in{\mathbb{R}^{n\times n}}$,  $A\geq_K B$ if $A-B\geq_K 0$. A matrix $A$ is called {\it K-monotone} if $A^{-1}$ exists and $A^{-1}\geq_K 0$ (see \cite{bpbook}). A vector $x\in \mathbb{R}^{n}$ is called {\it $K$-nonnegative ($K$-positive)} if $x\in K$ ($x\in$ int($K$)), and is denoted as $x\geq_K 0$ ($x>_K 0$). Similarly, for $x,y\in{\mathbb{R}^{n}}$,  $x\geq_K y$ ($x>_K y$) if $x-y\geq_K 0$ ($x-y>_K 0$). Applications of nonnegative matrices to ecology and epidemiology can be seen in the very recent article \cite{lewis} by Lewis {\it et al.}. Next results deal with nonnegativity of a matrix and its spectral radius. 

\begin{theorem}\textnormal{(Corollary 3.2 \& Lemma 3.3, \cite{marek:1990})}\label{l0}\\
Let $A\geq_K 0$. Then\\
$(i)$ $Ax\geq_K \alpha x$, $x\geq_K 0$, implies $\alpha \leq \rho(A)$. Moreover, if $Ax> \alpha x$, then $\alpha <\rho(A)$.\\
$(ii)$ $\beta x\geq_K Ax $, $x> 0$, implies $\rho(A)\leq \beta$. Moreover, if $ \alpha x> Ax$, then $\alpha >\rho(A)$.
\end{theorem}

\begin{theorem}\label{frob}\textnormal{(Theorem 1.3.2, \cite{bpbook} $\&$ Lemma 2, \cite{wang:2017})}\\
Let $A \geq_K 0$. Then\\
(i) $\rho(A)$ is an eigenvalue.\\
(ii) K contains an eigenvector of $A$ corresponding to $\rho(A)$.\\
(iii) $\rho(A)<\alpha$ if and only if $\alpha I-A$ is non-singular and $(\alpha I-A)^{-1}\geq_K 0.$
\end{theorem}

The next result discusses the convergence of a $K$-monotone sequence (i.e., a monotone sequence with respect to the proper cone $K$).

\begin{lemma}\textnormal{(Lemma 1, \cite{bpcones})}\label{conelem}\\
Let $K$ be a proper cone in ${\R}^n$ and let $\{s_{i}\}_{i=0}^{\infty}$ be a K-monotone non-decreasing sequence. Let $t \in {\R}^{n}$ be such that $t-s_{i} \in K$ for every positive integer $i.$ Then the sequence $\{s_{i}\}_{i=0}^{\infty}$ converges.
\end{lemma}

A comparison of the spectral radii of two different iteration matrices arising out of two matrix splittings
is useful for improving the speed of the iteration scheme \eqref{eqn1.2}. In this
direction, several comparison results have been introduced in the literature (see \cite{climent:1999, climent2:1999}). We recall below a few comparison results for the iterative scheme \eqref{eqn1.2} that are helpful to obtain our main results in Section \ref{MR}. The first two results stated below generalize Theorem 3.4 and Theorem 3.7 of \cite{woz:1994}, while the third one generalizes  Theorem 3.4 of \cite{nachi} for an arbitrary proper cone $K$. These results can be proved similarly as proved in \cite{woz:1994} and \cite{nachi} using our preliminary results, and is therefore omitted.

\begin{theorem}\label{2.10}
Let $A = U_1- V_1 = U_2-V_2$ be two $K$-weak regular splittings of type II of a
$K$-monotone matrix $A\in \mathbb{R}^{n\times n}$. If  $V_2\geq_K V_1$, then
 $\rho(U^{-1}_1V_1)\leq \rho(U^{-1}_2V_2)< 1$.
\end{theorem}

\begin{theorem}\label{2.9}
Let $A = U_1- V_1 = U_2-V_2$ be two $K$-weak regular splittings of different types of a
$K$-monotone matrix $A\in \mathbb{R}^{n\times n}$. If  $U^{-1}_{1}\geq_K U^{-1}_{2}$, then
 $\rho(U^{-1}_1V_1)\leq \rho(U^{-1}_2V_2)< 1$.
\end{theorem}

\begin{theorem}\label{2.12}
Let $A = U_1- V_1 = U_2-V_2$ be two $K$-weak regular splittings of type II of a
$K$-monotone matrix $A\in \mathbb{R}^{n\times n}$. If  $U_2\geq_K U_1\geq_K 0$, then
 $\rho(U^{-1}_1V_1)\leq \rho(U^{-1}_2V_2)< 1$.
\end{theorem}

\section{Main Results}\label{MR}

We divide this section into two parts. The first subsection discusses the convergence results for stationary and non-stationary two-stage method. We then classify the type of splitting induced by the two-stage iterative scheme. The second subsection discusses some interesting comparison results.

\subsection{Convergence Results}\label{CR}
In the case of standard proper cone $K=\mathbb{R}^{n}_{+}$, Frommer and Szyld \cite{szyld}  obtained the convergence criteria for stationary two-stage iteration scheme  \eqref{tsiter} in Theorem 4.3 \cite{szyld} when $U = F-G$ is a convergent weak regular splitting of type I.
In Theorem 4.4 \cite{szyld}, they stated the convergence result for non-stationary two-stage iteration scheme \eqref{tsiter}.
We state below the convergence result for stationary and non-stationary two-stage iteration schemes in an arbitrary proper cone setting.  We skip the proof as it follows similar steps as in \cite{szyld}.

\begin{theorem}\label{3.1}
Let $A=U-V$ be a convergent K-regular splitting and $U = F-G$ be a convergent K-weak regular splitting of type I. Then, the stationary and non-stationary  two-stage iteration scheme
is convergent for any sequence $s(k)\geq 1,~ k = 1, 2,\dots$ of inner iterations.
\end{theorem}

However, the convergence of \eqref{tsiter} is not yet studied if $U = F-G$ is not a weak regular splitting of type I even in the standard proper cone $\mathbb{R}^{n}_{+}$ setting. This issue is settled in this subsection for another class of splittings known as $K$-weak regular splitting of type II. To do this, 
we have 
\begin{align}\label{t}
T_{s(k)}=(F^{-1}G)^{s(k)}+\displaystyle\sum_{j=0}^{s(k)-1}(F^{-1}G)^{j}F^{-1}V=[I-(I-(F^{-1}G)^{s(k)})(I-U^{-1}V)]
\end{align}
from the two-stage iteration scheme \eqref{tsiter}. 
If the splitting $U=F-G$ for the system \eqref{inner} is a $K$-weak regular splitting of type II, then the matrix \begin{equation}\label{t'}
\widehat{T}_{s(k)}=(GF^{-1})^{s(k)}+\displaystyle\sum_{j=0}^{s(k)-1}(GF^{-1})^{j}VF^{-1}
\end{equation}
is $K$-nonnegative. 
Recall that two matrices $B$ and $C$ are {\it similar} if there exists a non-singular matrix $X$ such that $B=XCX^{-1}$. It is well-known that similar matrices have the same eigenvalues. Hence $\rho(B)=\rho(C)$. Based on this fact, we present below our first main result which says $T_{s(k)}$ and $\widehat{T}_{s(k)}$ have the same spectral radius under some assumption. 

\begin{lemma}\label{lem1}
Let $\widehat{T}_{s(k)}$ be as defined as \eqref{t'} and $T_{s(k)}$ be as defined as \eqref{t} for $s(k)=1,2,\dots$. If $VF^{-1}G=GF^{-1}V$, then $\rho(\widehat{T}_{s(k)})=\rho(T_{s(k)})$.
\end{lemma}
\begin{proof}
Since $VF^{-1}G=GF^{-1}V$, we have $V(F^{-1}G)^{s(k)}=(GF^{-1})^{s(k)}V$ for any nonnegative integer $s(k)$. Also, we observe that $U^{-1}GF^{-1}=F^{-1}GU^{-1}$. Therefore, $U^{-1}(GF^{-1})^{j}=(F^{-1}G)^{j}U^{-1}$ for any nonnegative integer $j$. Now,
\begin{align*}
AT_{s(k)}A^{-1}&=A[I-(I-(F^{-1}G)^{s(k)})(I-U^{-1}V)]A^{-1}\\
&=I-A(I-(F^{-1}G)^{s(k)})U^{-1}\\
&=I-A(U^{-1}-(F^{-1}G)^{s(k)}U^{-1})\\
&=I-AU^{-1}(I-(GF^{-1})^{s(k)})\\
&=I-(I-VU^{-1})(I-(GF^{-1})^{s(k)})\\
&=I-(I-(GF^{-1})^{s(k)})+VU^{-1}(I-(GF^{-1})^{s(k)})\\
\end{align*}
\begin{align*}
&=(GF^{-1})^{s(k)}+VU^{-1}\displaystyle \sum_{j=0}^{s(k)-1}(GF^{-1})^{j}(I-GF^{-1})\\
&=(GF^{-1})^{s(k)}+V\displaystyle \sum_{j=0}^{s(k)-1}(F^{-1}G)^{j}U^{-1}(I-GF^{-1})\\
&=(GF^{-1})^{s(k)}+\displaystyle \sum_{j=0}^{s(k)-1}(GF^{-1})^{j}VU^{-1}(I-GF^{-1})\\
&=(GF^{-1})^{s(k)}+\displaystyle \sum_{j=0}^{s(k)-1}(GF^{-1})^{j}VF^{-1}=\widehat{T}_{s(k)}.
\end{align*}
Thus, the matrices $T_{s(k)}$ and $\widehat{T}_{s(k)}$ are similar. Hence, $\rho(\widehat{T}_{s(k)})=\rho(T_{s(k)})$.
\end{proof}

Next, we establish the convergence of \eqref{tsiter} when the splitting $U=F-G$ is a $K$-weak regular splitting of type II that partially fulfills the objective of the paper.

\begin{theorem}\label{cgs}
Let $A=U-V$ be a convergent K-regular splitting and $U=F-G$ be a convergent K-weak regular splitting of type II such that $VF^{-1}G=GF^{-1}V$. Then, the stationary two-stage iterative method is convergent for any initial vector $x^{0}$. 
\end{theorem}
\begin{proof}
We have $\widehat{T}_{s(k)}\geq_K 0$ such that
\begin{align*}
 \widehat{T}_{s(k)}&=(GF^{-1})^{s(k)}+\displaystyle \sum_{j=0}^{s(k)-1}(GF^{-1})^{j}VF^{-1}\\
  &=(GF^{-1})^{s(k)}+\displaystyle \sum_{j=0}^{s(k)-1}(GF^{-1})^{j}VU^{-1}(I-GF^{-1})\\
  &=(GF^{-1})^{s(k)}+\displaystyle \sum_{j=0}^{s(k)-1}(GF^{-1})^{j}(VU^{-1}-VU^{-1}GF^{-1})\\
  &=(GF^{-1})^{s(k)}+\displaystyle \sum_{j=0}^{s(k)-1}(GF^{-1})^{j}(VU^{-1}-VF^{-1}GU^{-1})\\
  &=(GF^{-1})^{s(k)}+\displaystyle \sum_{j=0}^{s(k)-1}(GF^{-1})^{j}(VU^{-1}-GF^{-1}VU^{-1})\\
  &=(GF^{-1})^{s(k)}+\displaystyle \sum_{j=0}^{s(k)-1}(GF^{-1})^{j}(I-GF^{-1})VU^{-1}\\
  &=I-(I-(GF^{-1})^{s(k)})+(I-(GF^{-1})^{s(k)})VU^{-1}\\
   &=I-(I-(GF^{-1})^{s(k)})(I-VU^{-1})\\
    &=I-\displaystyle \sum_{j=0}^{s(k)-1}(GF^{-1})^{j}(I-GF^{-1})(I-VU^{-1}).
\end{align*}
Let $y>_K0$.  Then, $x=(I-VU^{-1})^{-1}(I-GF^{-1})^{-1}y>_K 0$. Now, post-multiplying $\widehat{T}_{s(k)}=I-\displaystyle \sum_{j=0}^{s(k)-1}(GF^{-1})^{j}(I-GF^{-1})(I-VU^{-1})$ by $x$, we get $\widehat{T}_{s(k)}x\geq_K 0$ such that
$x>_K x-\displaystyle \sum_{j=0}^{s(k)-1}(GF^{-1})^{j}y=\widehat{T}_{s(k)}x$. By Theorem \ref{l0} (ii), we have $\rho(\widehat{T}_{s(k)})<1$. Hence $\rho(T_{s(k)})<1$ by Lemma \ref{lem1}.
\end{proof}

In the standard proper cone setting ($K=\mathbb{R}^{n}_{+}$), we have the following new result.

\begin{corollary}
Let $A=U-V$ be a convergent regular splitting and $U=F-G$ be a convergent weak regular splitting of type II such that $VF^{-1}G=GF^{-1}V$. Then, the stationary two-stage iterative method is convergent for any initial vector $x^{0}$. 
\end{corollary}

\begin{remark}
Each of the result presented hereafter for the proper cone has an inbuilt corollary as mentioned above in the standard proper cone ($K=\mathbb{R}^{n}_{+}$) setting which is even a new result. 
\end{remark} 

For non-stationary two-stage method, we have the following result. The proof is similar to above, therefore we omit it.
\begin{theorem}\label{3.2}
Let $A=U-V$ be a convergent $K$-regular splitting and $U=F-G$ be a convergent $K$-weak regular splitting of type II such that $VF^{-1}G=GF^{-1}V$. Then, the non-stationary two-stage iterative method \eqref{tsiter} is convergent for any sequence $s(k)\geq 1,~k=0,1,2,\ldots$. 
\end{theorem}

Next result states that the matrices $T_{s(k)}$ and $\widehat{T}_{s(k)}$ induce the same splitting.

\begin{theorem}\label{impthm}
Let $A=U-V$ be a $K$-regular splitting of a $K$-monotone matrix $A \in {\R}^{n \times n}$. Let $U=F-G$ be a $K$-weak regular splitting of type II such that $VF^{-1}G=GF^{-1}V$. Then, the matrices $T_{s(k)}$ and $\widehat{T}_{s(k)}$ induce the same splitting $A=B-C$, where $B=A(I-T_{s(k)})^{-1}$. Further, the unique splitting $A=X-Y$ induced by the matrix $\widehat{T}_{s(k)}$ is also a $K$-weak regular splitting of type II.
\end{theorem}
\begin{proof}
We have $B=A(I-T_{s(k)})^{-1}$ and $C=B-A$. Let $X=(I-\widehat{T}_{s(k)})^{-1}A$ and $Y=X-A$. We will show that $T_{s(k)}$ and $\widehat{T}_{s(k)}$ induce the same splitting $A=B-C$. Since $\widehat{T}_{s(k)}=AT_{s(k)}A^{-1}$, so $\widehat{T}^{i}_{s(k)}=AT^{i}_{s(k)}A^{-1}$ for any nonnegative integer $i$. Now, $X=(I-\widehat{T}_{s(k)})^{-1}A=\displaystyle \sum_{j=0}^{\infty}\widehat{T}^{j}_{s(k)}A=\displaystyle \sum_{j=0}^{\infty}AT^{j}_{s(k)}A^{-1}A=A\displaystyle \sum_{j=0}^{\infty}T^{j}_{s(k)}=A(I-T_{s(k)})^{-1}=B$. Now,
\begin{align*}
    X^{-1}&=A^{-1}(I-\widehat{T}_{s(k)})\\
    &=A^{-1}(I-AT_{s(k)}A^{-1})\\
    &=(I-T_{s(k)})A^{-1}\\
    &=(I-(F^{-1}G)^{s(k)})(I-U^{-1}V)A^{-1}\\
    &=(I-(F^{-1}G)^{s(k)})U^{-1}\\
    &=U^{-1}(I-(GF^{-1})^{s(k)})\\
    &=F^{-1}(I-GF^{-1})^{-1}(I-(GF^{-1})^{s(k)})\\
    &=F^{-1}\displaystyle \sum_{j=0}^{s(k)-1}(GF^{-1})^{j}\geq_K 0.
\end{align*}
Also, $YX^{-1}=(X-A)X^{-1}=I-AA^{-1}(I-\widehat{T}_{s(k)})=\widehat{T}_{s(k)}\geq_K 0$. Thus, $A=X-Y$ is a $K$-weak regular splitting of type II. Let $A=X_1-Y_1$ be another splitting induced by $\widehat{T}_{s(k)}$ such that $\widehat{T}_{s(k)}=Y_{1}X_{1}^{-1}$. Then $A=X_{1}-\widehat{T}_{s(k)}X_{1}=(I-\widehat{T}_{s(k)})X_{1}$ which implies $X_{1}=(I-\widehat{T}_{s(k)})^{-1}A=X$. Hence, $A=X-Y$ is a unique $K$-weak regular splitting of type II induced by $\widehat{T}_{s(k)}$.
\end{proof}

\begin{remark}\label{rem1}
From the above result, it is easy to observe that the induced splitting has the form
\begin{equation*}
   A=P_{s(k)}-\widehat{T}_{s(k)}P_{s(k)}, 
\end{equation*}
where the matrix $\widehat{T}_{s(k)}$ is as defined by \eqref{t'} and $P_{s(k)}^{-1}=\displaystyle\sum_{j=0}^{s(k)-1}(F^{-1}G)^{j}F^{-1}$. Using the fact that $(F^{-1}G)^{j}F^{-1}=F^{-1}(GF^{-1})^{j}$ for any nonnegative integer $j$, we have $P_{s(k)}^{-1}=\displaystyle\sum_{j=0}^{s(k)-1}(F^{-1}G)^{j}F^{-1}=F^{-1}\sum_{j=0}^{s(k)-1}(GF^{-1})^{j}$. Thus, $P_{s(k)}^{-1}\geq_K 0$ whenever the splitting $U=F-G$ is a $K$-weak regular splitting of type II. Similarly, if $U=F-G$ is a $K$-weak regular splitting of type I in the above theorem,  then the induced splitting is a unique $K$-weak regular splitting of type I. While proving the same, we do not need the assumption $VF^{-1}G=GF^{-1}V$.
\end{remark}

In the following, we provide some sufficient conditions for the induced splitting to be a $K$-regular splitting.

\begin{theorem}\label{prsI}
Let $A=U-V$ be a $K$-regular splitting of a $K$-monotone matrix $A\in \mathbb{R}^{n\times n}$. Let $U=F-G$ be a $K$-weak regular splitting of type II such that $VF^{-1}G=GF^{-1}V$. If $G\geq_K GF^{-1}G$, then the induced splitting $A={P}_{s(k)}-\widehat{T}_{s(k)}{P}_{s(k)}$ is a $K$-regular splitting.
\end{theorem}
\begin{proof}
By Theorem \ref{impthm} amd Remark \ref{rem1}, the induced splitting $A={P}_{s(k)}-\widehat{T}_{s(k)}{P}_{s(k)}$ is a $K$-weak regular splitting of type II. So, we only need to prove that $\widehat{T}_{s(k)}{P}_{s(k)}\geq_K 0$. We have 
\begin{align*}
    \widehat{T}_{s(k)}{P}_{s(k)} & ={P}_{s(k)}-A\\
    &=(I-(GF^{-1})^{s(k)})^{-1}U-U+V\\
    &=[(I-(GF^{-1})^{s(k)})^{-1}-I]U+V\\
    &=(I-(GF^{-1})^{s(k)})^{-1}(GF^{-1})^{s(k)}U+V.
\end{align*}
Now, in order to show that $(GF^{-1})^{s(k)}U\geq_K 0$, it is sufficient to prove that $(GF^{-1})^2U \geq_K 0$. So,
\begin{align*}
    (GF^{-1})^2U&=(GF^{-1})(GF^{-1})U\\
    &=GF^{-1}GF^{-1}(F-G)\\
    &=GF^{-1}(G-GF^{-1}G)\geq_K 0.
\end{align*}
As $GF^{-1} \geq_K 0$ and $(GF^{-1})^2U \geq_K 0$, we have $(GF^{-1})^{s(k)}U \geq_K 0.$ Hence $\widehat{T}_{s(k)}{P}_{s(k)} \geq_K 0.$
\end{proof}


\subsection{Comparison Results}
 In this section, we prove certain comparison results. These results help us to choose a splitting that yields faster convergence of the respective two-stage iterative scheme \eqref{tsiter}. 
In this aspect, we now frame two different two-stage iterative schemes by taking two different matrix splittings $U=F-G=\overline{F}-\overline{G}$
whose corresponding iteration matrices are $T_{s(k)}=(F^{-1}G)^{s(k)}+\displaystyle\sum_{j=0}^{s(k)-1}(F^{-1}G)^{j}F^{-1}V$ and $\overline{T}_{s(k)}=(\overline{F}^{-1}~\overline{G})^{s(k)}+\displaystyle\sum_{j=0}^{s(k)-1}(\overline{F}^{-1}~\overline{G})^{j}\overline{F}^{-1}V$ with same number of inner iterations $s(k)$. But, when the two splittings $U=F-G=\overline{F}-\overline{G}$ are $K$-weak regular splittings of type II, then the matrices $\widehat{T}_{s(k)}=(GF^{-1})^{s(k)}+\displaystyle\sum_{j=0}^{s(k)-1}(GF^{-1})^{j}VF^{-1}$ and $\widehat{\overline{T}}_{s(k)}=(\overline{G}~\overline{F}^{-1})^{s(k)}+\displaystyle\sum_{j=0}^{s(k)-1}(\overline{G}~\overline{F}^{-1})^{j}V\overline{F}^{-1}$ are $K$-nonnegative. We use this information to prove our first comparison result presented below.

\begin{theorem}\label{3.11}
Let $A=U-V$ be a $K$-regular splitting of a $K$-monotone matrix $A \in {\R}^{n \times n}.$ Let $U=F - G=\overline{F}- \overline{G}$ be $K$-weak regular splittings of type II of a $K$-nonnegative matrix $U$ such that $VF^{-1}G=GF^{-1}V$ and $V\overline{F}^{-1}~\overline{G}=\overline{G}~\overline{F}^{-1}V$. If $\overline{G}~\overline{F}^{-1}\geq_K GF^{-1}$, then $\rho(T_{s(k)}) \leq \rho(\overline{T}_{s(k)})<1$.
\end{theorem}
\begin{proof}
By Theorem \ref{3.2}, we have $\rho(T_{s(k)})<1$ and $\rho(\overline{T}_{s(k)})<1$. Now, by Theorem \ref{impthm} and Remark \ref{rem1}, the induced splittings $A=P_{s(k)}-\widehat{T}_{s(k)}P_{s(k)}$ and $A=\overline{P}_{s(k)}-\widehat{\overline{T}}_{s(k)}\overline{P}_{s(k)}$ are $K$-weak regular splittings of type II. Since $\rho(F^{-1}G)<1$ and $\rho(\overline{F}^{-1}~\overline{G})<1$ as $U$ is $K$-monotone by Theorem \ref{2.5}, the condition $\overline{G}~\overline{F}^{-1}\geq_K GF^{-1}\geq_K 0$ implies that $[I-(\overline{G}~\overline{F}^{-1})^{s(k)}]^{-1}\geq_K [I-(G~F^{-1})^{s(k)}]^{-1}\geq_K 0$ by Theorem \ref{frob} (iii) which further yields $\widehat{\overline{T}}_{s(k)}\overline{P}_{s(k)}=[I-(\overline{G}~\overline{F}^{-1})^{s(k)}]^{-1}U-A\geq_K [I-(GF^{-1})^{s(k)}]^{-1}U-A=\widehat{T}_{s(k)}P_{s(k)}$. Thus, applying Theorem \ref{2.10} to the splittings $A=P_{s(k)}-\widehat{T}_{s(k)}P_{s(k)}$ and $A=\overline{P}_{s(k)}-\widehat{\overline{T}}_{s(k)}\overline{P}_{s(k)}$, we get $\rho(\widehat{T}_{s(k)})\leq \rho(\widehat{\overline{T}}_{s(k)})$. Hence $\rho(T_{s(k)})\leq \rho(\overline{T}_{s(k)})<1$ by Lemma \ref{lem1}.
\end{proof}

Note that the above result can also be proved using Theorem \ref{2.12}. Since $U$ is $K$-nonnegative and $\overline{G}~\overline{F}^{-1}\geq_K GF^{-1}$, we then have $\overline{P}_{s(k)}=[I-(\overline{G}~\overline{F}^{-1})^{s(k)}]^{-1}U \geq_K [I-(GF^{-1})^{s(k)}]^{-1}U=P_{s(k)}\geq_K 0$. Thus, $\rho(T_{s(k)}) \leq \rho(\overline{T}_{s(k)})$ by Theorem \ref{2.12}.
The $K$-nonnegative restriction on the matrix $U$ in Theorem \ref{3.11} can be dropped if we add the condition $G\geq_K GF^{-1}G$ to the above result. Next, we illustrate a few more comparison results.

\begin{theorem}\label{3.16}
Let $A=U-V$ be a $K$-regular splitting of a $K$-monotone matrix $A \in {\R}^{n \times n}$. Let $U=F-G=\overline{F}-\overline{G}$ be $K$-weak regular splittings of type II of $U$ such that $VF^{-1}G=GF^{-1}V$ and $V\overline{F}^{-1}~\overline{G}=\overline{G}~\overline{F}^{-1}V$. If $G\geq_K GF^{-1}G$ and $\overline{G}~\overline{F}^{-1}\geq_K GF^{-1}$, then $\rho(T_{s(k)}) \leq \rho(\overline{T}_{s(k)})<1$.
\end{theorem}

\begin{proof}
By Theorem \ref{impthm} and Theorem \ref{prsI}, the induced splittings $A=P_{s(k)}-\widehat{T}_{s(k)}P_{s(k)}$ and $A=\overline{P}_{s(k)}-\widehat{\overline{T}}_{s(k)}\overline{P}_{s(k)}$ are  $K$-regular and $K$-weak regular splitting of type II, respectively. Utilizing the inequality $\overline{G}~\overline{F}^{-1}\geq_K GF^{-1}$, we get $P^{-1}_{s(k)}=U^{-1}(I-(GF^{-1})^{s(k)})\geq_K U^{-1}(I-(\overline{G}~\overline{F}^{-1})^{s(k)})=\overline{P}^{-1}_{s(k)}$ which further implies $\rho(\widehat{T}_{s(k)}) \leq \rho(\widehat{\overline{T}}_{s(k)})$ by Theorem \ref{2.9}. Hence $\rho(T_{s(k)}) \leq \rho(\overline{T}_{s(k)})<1$ by Lemma \ref{lem1}.
\end{proof}

\begin{theorem}\label{3.10}
Let $A=U-V$ be a $K$-regular splitting of a $K$-monotone matrix $A \in {\R}^{n \times n}.$ Let $U=F - G=\overline{F}- \overline{G}$ be $K$-weak regular splittings of type II of $U$ such that $VF^{-1}G=GF^{-1}V$ and $V\overline{F}^{-1}~\overline{G}=\overline{G}~\overline{F}^{-1}V$. Then $\rho(T_{s(k)}) \leq \rho(\overline{T}_{s(k)})<1$, provided any one of the following conditions hold:\\
$(i)$ $\overline{G}~\overline{F}^{-1}\geq_K GF^{-1}$ and $\widehat{T}_{s(k)}P_{s(k)}\geq_K 0$,\\
$(ii)$ $\widehat{\overline{P}}_{s(k)}\widehat{\overline{T}}_{s(k)}\geq_K P_{s(k)}T_{s(k)}$. 
\end{theorem}

\begin{proof}
$(i)$ By the condition $\overline{G}~\overline{F}^{-1}\geq_K GF^{-1}$, we get $P^{-1}_{s(k)}\geq_K \overline{P}^{-1}_{s(k)}$ using the same argument as in Theorem \ref{3.16}. Now, $\widehat{T}_{s(k)}P_{s(k)}\geq_K 0$ and $P^{-1}_{s(k)}\geq_K \overline{P}^{-1}_{s(k)}$ implies $\rho(\widehat{T}_{s(k)})\leq \rho(\widehat{\overline{T}}_{s(k)})$ by Theorem \ref{2.9}. Hence, $\rho(T_{s(k)})\leq \rho(\overline{T}_{s(k)})<1$ by Lemma \ref{lem1}.\\
$(ii)$ Applying Theorem \ref{2.10} to the induced $K$-weak regular splittings of type II $A=P_{s(k)}-\widehat{T}_{s(k)}P_{s(k)}$ and $A=\overline{P}_{s(k)}-\widehat{\overline{T}}_{s(k)}\overline{P}_{s(k)}$, we directly obtain  $\rho(\widehat{T}_{s(k)}) \leq \rho(\widehat{\overline{T}}_{s(k)})$ which implies $\rho(T_{s(k)})\leq \rho(\overline{T}_{s(k)})<1$ by Lemma \ref{lem1}.
\end{proof}

\begin{theorem}
Let $A=U-V$ be a $K$-regular splitting of a $K$-monotone matrix $A \in {\R}^{n \times n}.$ Let $U=F -G=\overline{F}-\overline{G}$ be $K$-weak regular splittings of type II of $U$ such that $VF^{-1}G=GF^{-1}V$ and $V\overline{F}^{-1}~\overline{G}=\overline{G}~\overline{F}^{-1}V$. Then $\rho(T_{s(k)}) \leq \rho(\overline{T}_{s(k)})<1,$ provided the following conditions hold:\\
$(i)$ $\widehat{\overline{T}}_{s(k)}\overline{P}_{s(k)} \geq_K 0$, $k=0,1,2,\ldots$, \\
$(ii)$ $F^{-1}\geq_K \overline{F}^{-1}$,\\
$(iii)$ $\overline{F}^{-1}~\overline{G}\geq_K 0$.
\end{theorem}
\begin{proof}
 Since $U^{-1}GF^{-1}=U^{-1}(F-U)F^{-1}=U^{-1}-F^{-1}\geq_K 0$, utilizing condition (ii), we get $U^{-1}-\overline{F}^{-1}\geq_K U^{-1}-F^{-1}$.  Using condition (iii), we also observe that $U^{-1}(\overline{G}~\overline{F}^{-1})^{k-1}U=\overline{F}^{-1}(I-\overline{G}~\overline{F}^{-1})^{-1}(\overline{G}~\overline{F}^{-1})^{k-1}(I-\overline{G}~\overline{F}^{-1})\overline{F}
=\overline{F}^{-1}(\overline{G}~\overline{F}^{-1})^{k-1}\overline{F}
=(\overline{F}^{-1}~\overline{G})^{k-1} \geq_K 0$. We will now use the method of induction to show that $U^{-1}(\overline{G}~\overline{F}^{-1})^{k}\geq_K U^{-1}(GF^{-1})^{k},~k=0,1,2, \ldots$ is true. For $k=0$, the inequality $U^{-1}(\overline{G}~\overline{F}^{-1})^{k}\geq_K U^{-1}(GF^{-1})^{k} $ is trivial. Suppose that the inequality holds for $k=0,1,\ldots,p$. Then, for $k=p+1,$ we have
\begin{align*}
U^{-1}(\overline{G}~\overline{F}^{-1})^{p+1} 
&=U^{-1}\overline{G}~\overline{F}^{-1}(\overline{G}~\overline{F}^{-1})^{p}\\
&=(U^{-1}-\overline{F}^{-1})(\overline{G}~\overline{F}^{-1})^{p}\\
& \geq_K (U^{-1}-F^{-1})(\overline{G}~\overline{F}^{-1})^{p}\\
&= (U^{-1}GF^{-1})(\overline{G}~\overline{F}^{-1})^{p}\\
&= F^{-1}GU^{-1}(\overline{G}~\overline{F}^{-1})^{p}\\
&\geq_K F^{-1}GU^{-1}(GF^{-1})^{p}\\
&= U^{-1}(GF^{-1})^{p+1}.
\end{align*}
So, $U^{-1}(\overline{G}~\overline{F}^{-1})^{k}\geq_K U^{-1}(GF^{-1})^{k} $ holds for all $k=0,1,2,\ldots$ which implies that $U^{-1}-U^{-1}(GF^{-1})^{k}\geq_K U^{-1}-U^{-1}(\overline{G}~\overline{F}^{-1})^{k}$ for all $k=0,1,2,\ldots$, i.e., $P_{s(k)}^{-1}\geq_K {\overline{P}}_{q(k)}^{-1}$, $k=0,1,2, \ldots$. By Theorem \ref{2.9}, we thus have  $\rho(\widehat{T}_{s(k)}) \leq \rho(\widehat{\overline{T}}_{s(k)})$. Hence, $\rho(T_{s(k)}) \leq \rho(\overline{T}_{s(k)})$ by Lemma \ref{lem1}.
\end{proof}

We end this subsection with a result that compares two $K$-weak regular splittings of different types.

\begin{theorem}\label{3.9}
Let $A=U-V$ be a $K$-regular splitting of a $K$-monotone matrix $A \in {\R}^{n \times n}.$ Let $U=F - G$ be $K$-weak regular splitting of type I and $U=\overline{F}- \overline{G}$ be a $K$-weak regular splitting of type II of $U$ such that $V\overline{F}^{-1}~\overline{G}=\overline{G}~\overline{F}^{-1}V$. If $F^{-1}\geq_K \overline{F}^{-1}$ and $F^{-1}G\geq_K \overline{F}^{-1}~\overline{G}$, then $\rho(T_{s(k)}) \leq \rho(\overline{T}_{s(k)})<1$.
\end{theorem}

\begin{proof}
     Applying Theorem \ref{3.1} and Theorem \ref{3.2}, we have $\rho(T_{s(k)})<1$ and $\rho(\overline{T}_{s(k)})<1$, respectively. The induced splittings $A=P_{s(k)}-P_{s(k)}T_{s(k)}$ is a $K$-weak regular splitting of type I and $A=\overline{P}_{s(k)}-\widehat{\overline{T}}_{s(k)}\overline{P}_{s(k)}$ is a $K$-weak regular splitting of type II by  Theorem \ref{impthm} and Remark \ref{rem1}. Since $F^{-1}G\geq_K \overline{F}^{-1}~\overline{G}\geq_K 0$, we have $(F^{-1}G)^{j}\geq_K (\overline{F}^{-1}~\overline{G})^{j}$ for any nonnegative integer $j$. Therefore, $\displaystyle \sum_{j=0}^{s(k)-1}(F^{-1}G)^{j}\geq_K \sum_{j=0}^{s(k)-1}(\overline{F}^{-1}~\overline{G})^{j}$. Now, using $\displaystyle \sum_{j=0}^{s(k)-1}(F^{-1}G)^{j}\geq_K \sum_{j=0}^{s(k)-1}(\overline{F}^{-1}~\overline{G})^{j}$ and the condition $F^{-1}\geq_K \overline{F}^{-1}\geq_K 0$, we get $P^{-1}=\displaystyle \sum_{j=0}^{s(k)-1}(F^{-1}G)^{j}F^{-1}\geq_K \sum_{j=0}^{s(k)-1}(\overline{F}^{-1}~\overline{G})^{j}\overline{F}^{-1}=\overline{F}^{-1}\sum_{j=0}^{s(k)-1}(\overline{G}~\overline{F}^{-1})^{j}=\overline{P}^{-1}_{s(k)}$. We thus obtain $\rho(T_{s(k)})\leq \rho(\widehat{\overline{T}}_{s(k)})$ by Theorem \ref{2.9}. Hence, $\rho(T_{s(k)})\leq \rho(\overline{T}_{s(k)})$ by Lemma \ref{lem1}.
\end{proof}

\section{Monotone Iterations}\label{monotone}
In this section, we  discuss the monotone convergence theory of the two-stage stationary iterative method (\ref{tsiter}). The monotone convergence theorem for the case when $U=F-G$ is a weak regular splitting of type I was proved by Bai \cite{bai}. We prove the case when $U=F-G$ is a $K$-weak regular splitting of type II.  To this end,  we need an additional assumption ``$A$ is $K$-nonnegative", and the same is shown hereunder. 

\begin{theorem}\label{thm4.1}\textnormal{(Monotone Convergence Theorem)} Let $A=U-V$ be a $K$-regular splitting of a $K$-nonnegative and $K$-monotone matrix $A\in \mathbb{R}^{n\times n}$. Further, assume that $U=F-G$ be a $K$-weak regular splitting of type II such that $VF^{-1}G=GF^{-1}V$ and $s(k) \geq 1$, $k=0,1,2\ldots$ be the inner iteration sequence. If $x_{0}$ and $y_{0}$ are initial values that hold
\begin{equation}\label{ineq}
x_1\geq_K x_0,~ y_0\geq_K y_1 \textnormal{ and } y_0\geq_K A^{-1}b \geq_K x_{0}.
\end{equation} Then, the sequences $\{x_{k}\}$ and $\{y_{k}\}$ generated by 
$$x_{k+1}=T_{s(k)}x_{k}+P_{s(k)}^{-1}b,$$
$$y_{k+1}=T_{s(k)}y_{k}+P_{s(k)}^{-1}b,$$
$k=0,1,2,\ldots$ satisfy

$(i)$ $y_{k} \geq_K y_{k+1}\geq _K x_{k+1} \geq_K x_{k},~k=0,1,2,\ldots$,

$(ii)$ $\displaystyle \lim_{n \rightarrow \infty} x_{k}=A^{-1}b=\displaystyle \lim_{k \rightarrow \infty} y_{k}$ and $y_{k} \geq_K y_{k+1}\geq _K A^{-1}b \geq_K x_{k+1} \geq_K x_{k},~k=0,1,2,\ldots$. 
\end{theorem}
\begin{proof}
$(i)$ We will show by induction that $x_{k+1} \geq_K x_{k}$ for $k=0,1,2,\ldots$. The case $k=0$ is established by the hypothesis. Assume that the result holds for $k=0,1,\ldots, p>0$ so that $x_{p+1}-x_{p} \geq_K 0$, then there exist $z_{p+1}$ and $z_{p}$ such that $z_{p+1}-z_{p}=A(x_{p+1}-x_{p}) \geq_K 0$. Since $s(k)$ is independent of $k$ and $\widehat{T}_{s(k)}\geq_K 0$ for $k=0,1,\ldots$, we have
\begin{align*}
   A^{-1}\widehat{T}_{s(p)}(z_{p+1}-z_{p})&= T_{s(p)}A^{-1}(z_{p+1}-z_{p})\\
   &=T_{s(p)}(x_{p+1}-x_{p})\\
   &=(T_{s(p)}x_{p+1}+\displaystyle\sum_{j=0}^{s(p)-1}(F^{-1}G)^{j}F^{-1}b)-(T_{s(p)}x_{p}+\displaystyle\sum_{j=0}^{s(p)-1}(F^{-1}G)^{j}F^{-1}b)\\
   &=(T_{s(p+1)}x_{p+1}+\displaystyle\sum_{j=0}^{s(p+1)-1}(F^{-1}G)^{j}F^{-1}b)-(T_{s(p)}x_{p}+\displaystyle\sum_{j=0}^{s(p)-1}(F^{-1}G)^{j}F^{-1}b)\\
   &=x_{p+2}-x_{p+1}\geq_K 0.
\end{align*}
Similarly, we can show that $y_{p} \geq_K y_{p+1}$ for each $p$. Now, assume that $y_{p}-x_{p}\geq_K 0$ for some $p>0$, then
\begin{align*}
    y_{p+1}-x_{p+1}
    &=(T_{s(p)}y_{p}+P_{s(p)}^{-1}b)-(T_{s(p)}x_{p}+P_{s(p)}^{-1}b)\\
    &=T_{s(p)}(y_{p}-x_{p})\\
    &=A^{-1}\widehat{T}_{s(p)}A(y_{p}-x_{p})\geq_K 0.
\end{align*}
Again, it follows by induction that $y_{p} \geq_K x_{p}$ for each $p$. Thus, $y_{k} \geq_K y_{k+1}\geq _K x_{k+1} \geq_K x_{k},~k=0,1,2,\ldots$.

$(ii)$ The sequence $\{x_{k}\}$ is $K$-monotonic increasing and there exists $y_0 \in \mathbb{R}^{n}$ such that $y_0-x_k\geq_K 0$ for all $k$, therefore it converges by Lemma \ref{conelem}.
Similarly, $\{-y_{k}\}$ is $K$-monotonic increasing and there exists $-x_0 \in \mathbb{R}^{n}$ such that $-x_0+y_{k}\geq_K 0$ for all $k$, therefore it converges by Lemma \ref{conelem}. This implies that the sequence $\{y_{k}\}$ also converges.
 Thus, the sequences $\{x_{k}\}$ and $\{y_{k}\}$ converge to $(I-T_{s(k)})^{-1}P^{-1}_{s(k)}b=P^{-1}_{s(k)}(I-\widehat{T}_{s(k)})^{-1}b$, i.e., $A^{-1}b$. Hence, $y_{k} \geq_K y_{k+1}\geq _K A^{-1}b \geq_K x_{k+1} \geq_K x_{k},~k=0,1,2,\ldots$.
\end{proof}

The existence of $x_{0}$ and $y_{0}$ which satisfies the inequality \eqref{ineq} is guaranteed by the following result.

\begin{theorem}
Let $A=U-V$ be a $K$-regular splitting and $U=F-G$ be a $K$-weak regular splitting of type II of a $K$-monotone matrix $A$ such that $VF^{-1}G=GF^{-1}V$. If $\rho(T_{s(k)})<1,$ then the existence of $x_{0}$ and $y_{0}$ are assured.
\end{theorem}
\begin{proof}
Assume that $\rho(T_{s(k)})<1$, then $\rho(\widehat{T}_{s(k)})<1$ by Lemma \ref{lem1}. Since $\widehat{T}_{s(k)} \geq_K 0,$ there exists $x \geq_K 0$ such that $\widehat{T}_{s(k)}x=\rho(T_{s(k)})x$ by Theorem \ref{frob} (i) \& (ii). Let $z=A^{-1}x$. Then, $z\geq_K 0$ and $T_{s(k)}z=\rho(T_{s(k)})z$. Therefore, $z\geq_K \rho(T_{s(k)})z$. Let $x_{0}=A^{-1}b-z.$ 
We so have
\begin{align*}
x_{1}&=T_{s(k)}x_{0}+\displaystyle\sum_{j=0}^{s(k)-1}(F^{-1}G)^{j}F^{-1}b\\
&=T_{s(k)}A^{-1}b+\displaystyle\sum_{j=0}^{s(k)-1}(F^{-1}G)^{j}F^{-1}b-T_{s(k)}z\\
&=[I-(I-(F^{-1}G)^{s(k)})(I-U^{-1}V)]A^{-1}b+\displaystyle\sum_{j=0}^{s(k)-1}(F^{-1}G)^{j}F^{-1}b-\rho(T_{s(k)})z\\
&=[I-(I-(F^{-1}G)^{s(k)})(I-U^{-1}V)](I-U^{-1}V)^{-1}U^{-1}b+\displaystyle\sum_{j=0}^{s(k)-1}(F^{-1}G)^{j}F^{-1}b-\rho(T_{s(k)})z\\
&=(I-U^{-1}V)^{-1}U^{-1}b-(I-(F^{-1}G)^{s(k)})U^{-1}b+\displaystyle\sum_{j=0}^{s(k)-1}(F^{-1}G)^{j}F^{-1}b-\rho(T_{s(k)})z\\
&=(I-U^{-1}V)^{-1}U^{-1}b-\displaystyle\sum_{j=0}^{s(k)-1}(F^{-1}G)^{j}F^{-1}b+\displaystyle\sum_{j=0}^{s(k)-1}(F^{-1}G)^{j}F^{-1}b-\rho(T_{s(k)})z\\
&=A^{-1}b-\rho(T_{s(k)})z\\
&\geq_K A^{-1}b-z=x_{0}.
\end{align*}
Setting $$y_{0}=A^{-1}b+z,$$
it then follows similarly that $y_{0} \geq_K y_{1}.$ Moreover, $y_{0}-x_{0}=2z \geq_K 0.$
\end{proof}

We conclude this section with the remark that if we consider the iteration scheme
$$X_{k+1}=U^{-1}VX_k+U^{-1}B,~ k= 0,1,2,\cdots,$$
then this scheme will converge to  $A^{-1}B$ for any initial matrix $X_0$ if and only if $\rho(U^{-1}V)<1.$
Analogously, the above discussed two-stage technique is also applicable to solve $AX=B$. Especially, the system with multiple right-hand side vectors, the splitting algorithms are advantageous as we need only one splitting for the entire computations and exactly two splittings for the two-stage iteration method.

\section{COVID-19 Pandemic Model \& Next Generation Matrix}\label{covid}
The pandemic model is localized, and is highly heterogeneous corresponding to the age structure and the different stages of disease transmission. A generalized pandemic model considers a heterogeneous population(intra-compartmental) that can be grouped into $n$ homogeneous compartments(inter-compartmental). Our focus is to identify the next generation matrix which involves the inverse of an $M$-matrix \cite{bpbook} in it. We are going to emphasize on an efficient numerical method to find the inverse of this special matrix. For the shake of completeness, the next generation matrix (NGM) is crucial in computing the reproduction number of the pandemic. The basic reproductive number $(R_0)$ of COVID-19 has been initially estimated by the World Health Organization (WHO) that
ranges between 1.4 and 2.5, as declared in the statement regarding the outbreak of SARS-CoV-2, dated January 23, 2020. Later in \cite{R0range:2020, R0covid:2020}, the researchers estimated that the mean of $R_0$ is higher than 3.28, and the median is higher than 2.79, by observing the super spreading nature and the doubling rate of this novel Coronavirus.
\begin{definition}\textnormal{(\cite{Rzero:1990})}\label{defR0}\\
In epidemiology, we take \emph{basic reproduction number/ratio},
$R_0$, as the average number of individuals infected by 
the single infected individual during his or her entire
infectious period, in a population which is entirely
susceptible.
\end{definition}
The basic reproduction number is a key parameter in the mathematical modeling of transmissible diseases. Very recently, Khajanchi and Sarkar \cite{SAIUQR:2020} considered a  compartmental model design to predict the possible infections in the COVID-19 pandemic in India. The model considers six compartment of populations susceptible(S), asymptomatic($I^a$), reported symptomatic($I^s$), unreported symptomatic($I^u$) and recovered($R$). This is called {\it SAIUQR pandemic model} and the same model is reproduced below.
\begin{align}\label{model}\nonumber
dS(t) &=  \mu - \beta S \bigg( \alpha_a \frac{I^a}{N} + \alpha_i \frac{I^s}{N} + \alpha_u \frac{I^u}{N} \bigg) + \rho_s\gamma_q Q -\delta S\\\nonumber
dI^a (t) &= \beta S \bigg( \alpha_a \frac{I^a}{N} + \alpha_i \frac{I^s}{N} + \alpha_u \frac{I^u}{N} \bigg) - (\xi_a+\gamma_a)I^a-\eta_a I^a-\delta I^a + \phi R\\
dI^s(t) &= \theta \gamma_a I^a + (1-\rho_s)\gamma_q Q  -\eta_i I^s -\delta I^s\\\nonumber
dI^u(t) &= (1-\theta) \gamma_a I^a -\eta_u I^u -\delta I^u\\\nonumber
dQ(t) &= \xi_a I^a -\gamma_q Q - \delta Q\\\nonumber
dR(t) &=  \eta_a I^a + \eta_i I^s + \eta_u I^u  - \phi R - \delta R
\end{align}
To be precise, the solutions to the above system of differential equations leave invariant a certain cone in $\mathbb{R}^n$, where $n$ is the number of compartments. Our mathematical model introduces some demographic effects by assuming a proportional natural mortality rate of $\delta > 0$ and birth rate $\mu$ per unit time. The parameter $\beta$
represents the probability of disease transmission rate. Let $\alpha_a$, $\alpha_s$, and $\alpha_u$ be the adjustment
factors with the disease transmission rate. A quarantined population can either move to the susceptible or infected compartment at the rate of $\rho_s$. Here, $\gamma_q$ is the rate at which the quarantined uninfected contacts are released into the wider community.  The asymptomatic
individuals deplete by reported and
unreported symptomatic individuals at the rate $\gamma_a$ with
a portion $\theta \in (0, 1)$, and become quarantine at the rate
$\xi_a$. Further, $\eta_a$, $\eta_i$ and $\eta_u$ are the recovery rate from the asymptomatic, the reported-symptomatic and the unreported-symptomatic class. A small modification to the existing model is by considering, some people return from the recovery class, again to the exposed class at the rate of $\phi$.
We have the following matrix $\mathcal{B}$ corresponding to the new infection
\begin{align*}
\mathcal{B} = \begin{pmatrix}
\beta\alpha_a & \beta\alpha_i &\beta \alpha_u &0\\
0&0&0&0\\
0&0&0&0\\
0&0&0&0
\end{pmatrix}.
\end{align*}

\noindent This matrix is of rank one for the present model but this can be of higher rank (for example: vector-host Model or two strain model). And one can see that this is a nonlinear matrix function of time \cite{NGM:2008}.  The matrix associated with the transition terms in the model  is
\begin{align*}
\mathcal{A} = \begin{pmatrix}
\xi_a+\gamma_a+\eta_a+\delta & 0 &0 & -\phi\\
-\theta \gamma_a&\eta_i+\delta&0&-(1-\rho_s)\gamma_q\\
-(1-\theta) \gamma_a&0&\eta_u+\delta&0\\
-\xi_a&0&0&\gamma_q +\delta
\end{pmatrix}.
\end{align*}
Here, the matrix $\mathcal{A}$ is always an $M$-matrix.
Finally, the next-generation matrix is defined as $\mathcal{B}\mathcal{A}^{-1}$ to compute the pandemic reproduction number
 $R_0 = \rho(\mathcal{B}\mathcal{A}^{-1})$. For more details about this threshold number and the special matrix, one can refer \cite{NGM:2008}. It's important to note that these matrices are larger than the $4 \times 4$ matrix which we have seen so far, in most of the realistic model.\\

\noindent The model can be modified to understand the impact of social distancing and lockdown measures on the entire pandemic growth like the model considered in \cite{IMSC:2020} for predicting the spread of COVID-19 in India. In this model, a social contact matrix is considered and is partitioned into the home, workplace, school and all other contacts. Our notation is $C$ for the entire contact matrix partitioned by workplace ($C^W$), home ($C^H$), school ($C^S$) and others ($C^O$). Thus, $C = C^W + C^H + C^S + C^O$, where the total contact can be reduced by controlling all parts except home contact. The lockdown and social distancing like interventions can be incorporated by multiplying a time-dependent control function with the respective contact. The time-dependent social contact matrix at a time is
\begin{align}
C_{ij}(t) = C^H_{ij} + u^W (t)C^W_{ij} + u^S (t)C^S_{ij} + u^O (t) C^O_{ij}
\end{align}
where $u^W (t)$, $u^S(t)$ and $u^O(t)$ are the control functions corresponding to contact matrices for work, school and others, depending on the percentage of lockdown implemented on their contacts.\\

\noindent Further, we consider the age structure of the population, and divide the population aggregated by age into $M$ groups labeled by $i = 1, 2, \cdots M$.
The population within the age group $i$ is partitioned into susceptible $S_i$, asymptomatic infectives $I^a_i$, reported symptomatic $I^s_i$, unreported symptomatic $I^u_i$ and removed individuals $R_i$. The sum
of these is the size of the population in age group $i$,
$N_i = S_i + I^a_i + I^s_i + I^u_i + R_i$. Therefore, the total population size is
$$\displaystyle{\sum_{i = 1}^M} N_i = N.$$
The contact matrix based on a demographic survey is suggested in \cite{contact:age}  that considers $16(=M)$  different age groups ranging from 1 to 80 age people. So, we have the contact matrix of order 16  with  $n$ number of disease transformation variables.
Then, the incidence function associated with  the depletion from susceptible class due to infected individuals is
$$\lambda = \beta \bigg( \alpha_a \frac{I^a}{N} + \alpha_i \frac{I^s}{N} + \alpha_u \frac{I^u}{N} \bigg).$$
This is modified by incorporating the contact matrix and age structure as follows:
$$\lambda_i(t) = \beta(t) \sum_{j=1}^M \bigg(C_{i\,j}^a(t)\dfrac{I^a_j}{N_j} +  C_{i\,j}^s(t)\dfrac{I^s_j}{N_j} +  C_{i\,j}^u(t) \dfrac{I^u_j}{N_j}\bigg),\; $$

\noindent where $C_{i\,j}^a(t)$, $C_{i\,j}^s(t)$ and $C_{i\,j}^u(t)$ are the fraction of the total contact matrix $C_{i\,j}(t)$ corresponding to the faction parameters $f_a$, $f_s$ and $f_u$, respectively. To find the reproduction number, we linearised  the dynamical system (\ref{model}) and evaluate the corresponding next generation matrix at the disease free fixed point $\Big(\dfrac{\mu}{\delta}, 0, 0, 0, 0\Big)$. Incorporating the $M$ age group and their social contacts, we have the required matrices
\begin{align}\label{Bkron}
    \mathcal{B} = \begin{pmatrix}
\beta\alpha_a & \beta\alpha_i &\beta \alpha_u & 0\\
0&0&0&0\\
0&0&0&0\\
0&0&0&0
\end{pmatrix}\otimes K
\end{align}
and
\begin{align}\label{Akron}
\mathcal{A} = \begin{pmatrix}
\xi_a+\gamma_a+\eta_a+\delta & 0 &0 & -\phi\\
-\theta \gamma_a&\eta_i+\delta&0&-(1-\rho_s)\gamma_q\\
-(1-\theta) \gamma_a&0&\eta_u+\delta&0\\
-\xi_a&0&0&\gamma_q +\delta
\end{pmatrix} \otimes \textbf{I}_M,
\end{align}
where $\otimes$ is the kronecker product and $\displaystyle{K_{i\,j} = \frac{C_{i\,j} N_{i}}{N_j}, \,(1 \leq i, \,j \leq M)}.$ The matrices $\mathcal{A}$ and $\mathcal{B}$ are now of order $64$, but this can be even bigger than  10,000 for larger data sets. For simplicity, we assume the social contact only in the same age group so that  $K$ reduces to the identity matrix. The matrices $\mathcal{A}$ and $\mathcal{B}$ are block diagonal matrices, and each block diagonal can be different if the model parameters vary with respect to age groups.

\subsection{Numerical Algorithm $\&$ Computations}
Motivated by the wide range of applications of the two-stage type iterative algorithm including the fast algorithm for the PageRank problem \cite{miga1}, more general Markov chain \cite{miga} and the Influence Maximization problems in social networks \cite{He}, we provide below the two-stage algorithm that we use for our computations.

\begin{algorithm}
  \caption{Two-stage Iteration}\label{alg1}
  \begin{algorithmic}[1]
      \Procedure{two-stage}{$\epsilon, s_k, A, b$}\Comment{$\epsilon$ = Tolerance and $s_k$
      = No. of inner iteration}
    \State \texttt{Generate a regular splitting } $A = U-V$
      \State \texttt{Initial guess} $x_0$
      \While{$\| x_n - x_{n+1}\| < \epsilon $}\Comment{Convergence condition}
        \State $y_0 = y = (y^{(1)}, y^{(2)}, \cdots, y^{(q)})\gets x_n$
        \For{\texttt{$i$ = 1 to $q$}}
         \State \texttt{Generate a weak regular splitting } $U = F-G$
            \For{\texttt{$j$ = 1 to $s_{k}-1$}}
            \State $F y_{j+1}^{(i)} \gets G y_j^{(i)} + \left( V y_j^{(i)} + b \right)^{i}$
            \EndFor
            \State $ x_{n+1}^{(i)} \gets y_{s_{k}}^{(i)}$
        \EndFor
        \State $x_{n+1} \gets (x_{n+1}^{(1)}, x_{n+1}^{(2)}, \cdots, x_{n+1}^{(q)})$
      \EndWhile\label{euclidendwhile}
        \EndProcedure
  \end{algorithmic}
\end{algorithm}

The model parameters are mostly estimated based on the data available from the COVID-19 spread during the first few days in India. Let us consider a particular set of data experimented in \cite{SAIUQR:2020}, the initial population sizes are $$(S, A, Q, I, U, R) =  (39402, 1500, 2000, 20, 0, 0)$$ for a particular state in India. The model parameters are $\mu =  1200$,  $\beta =  1.10$, $\alpha_a =  0.264$, $\alpha_i = 0.76$, $\alpha_u =  0.96$, $\xi_a =  0.07151$, $\gamma_a = 0.0012$, $\gamma_q = 0.0015$, $\delta = 0.03$, $\eta_a = 1/7.48$, $\theta = 0.8$, $\eta_i = 1/7$, $\eta_u = 1/7$,  $\rho_s = 0.5$, and $\phi = 0.1$, as per the prescribed data in \cite{SAIUQR:2020}. The prescribed data provides us the new infection matrix and disease transition matrix as follows:
\begin{align}\label{Bnum}
    \mathcal{B} &=\begin{bmatrix}
    0.2904 & 0.836 & 1.056 & 0\\
    0 & 0 & 0 & 0\\
    0 & 0 & 0 & 0\\
    0 & 0 & 0 & 0
\end{bmatrix}\\\label{Anum}\nonumber
\noindent \mbox{and} \\
\mathcal{A}&=\begin{bmatrix}
0.23639984 & 0 & 0 & -0.07\\
-0.00096 & 0.17285714 & 0 & -0.00075\\
-0.00024 & 0 & 0.17285714 & 0\\
-0.07151 & 0 & 0 & 0.0315
\end{bmatrix}.
\end{align}

\noindent The matrix $\mathcal{A}$ is an M-matrix and its inverse is computed using Matlab command $\mathcal{A} \backslash \mathcal{I}$. Here,
$$\mathcal{A}^{-1}=\begin{bmatrix}
1.06564745\times {10}^{+02} & 0 & 0 & 3.38300777\times {10}^{+02}\\
1.64147870 & 5.78512397 & 0 & 5.34878455\\
1.47957662\times {10}^{-01} & 0 & 5.78512397 & 4.69706864\times {10}^{-01}\\
2.41918885\times {10}^{+02} & 0 & 0 & 7.99742493\times {10}^{+02}
\end{bmatrix},$$
and the corresponding Next Generation Matrix is 
$$\mbox{\bf NGM}= \mathcal{B}\mathcal{A}^{-1}=\begin{bmatrix}
2.84091508\times {10}^{+01} & 2.25355931\times {10}^{-03} & 0 & 9.01878340\times {10}^{+01}\\
0 & 0 & 0 & 0\\
0 & 0 & 0 & 0\\
0 & 0 & 0 & 0
\end{bmatrix}.$$
Finally, we have the basic reproduction number $R_0 = \rho(\mathcal{B}\mathcal{A}^{-1}) = 3.9327471467109305$. As we have $\rho(\mathcal{B}\mathcal{A}^{-1}) = \rho(\mathcal{A}^{-1}\mathcal{B})$, so instead of computing $\mathcal{B}\mathcal{A}^{-1}$, we can compute $\mathcal{A}^{-1}\mathcal{B}$ to meet our purpose. Our aim is to compute the solution matrix $\mathcal{A}^{-1}\mathcal{B}$ for solving the matrix equation 
\begin{align}\label{multirhs}
    \mathcal{A}\mathcal{X} = \mathcal{B}
\end{align}
using two stage iterative method as discussed in Section \ref{MR}.

\begin{figure}[H]
	\centering
	~~(a)~~~~~~~~~~~~~~~~~~~~~~~~~~~~~~~~~~~~~~~~~~~~~~~~~~~~~~~~~~~(b)~~~\\[1ex]
	\includegraphics[height=2.8in,width=3.1in]{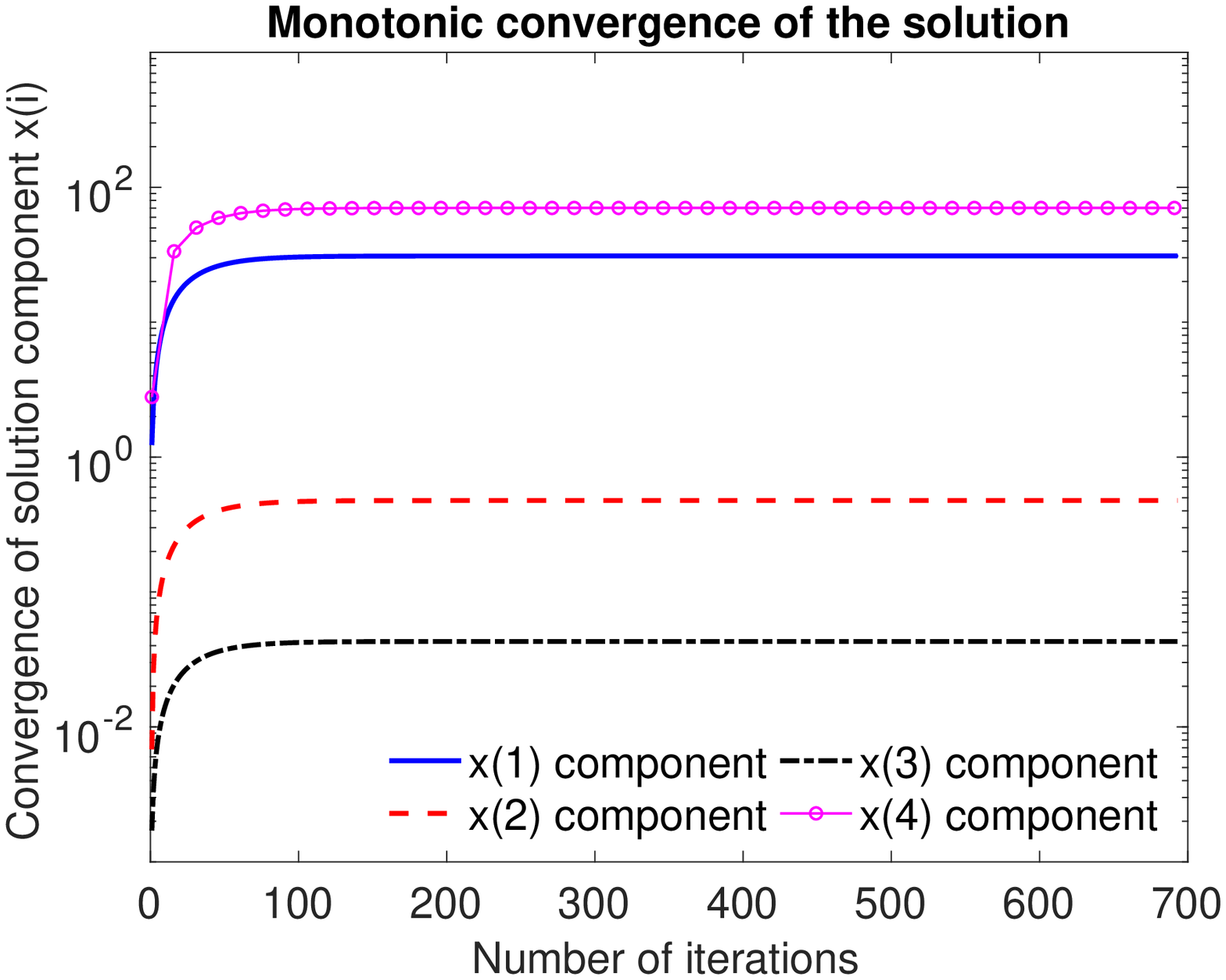}
	\includegraphics[height=2.8in,width=3.1in]{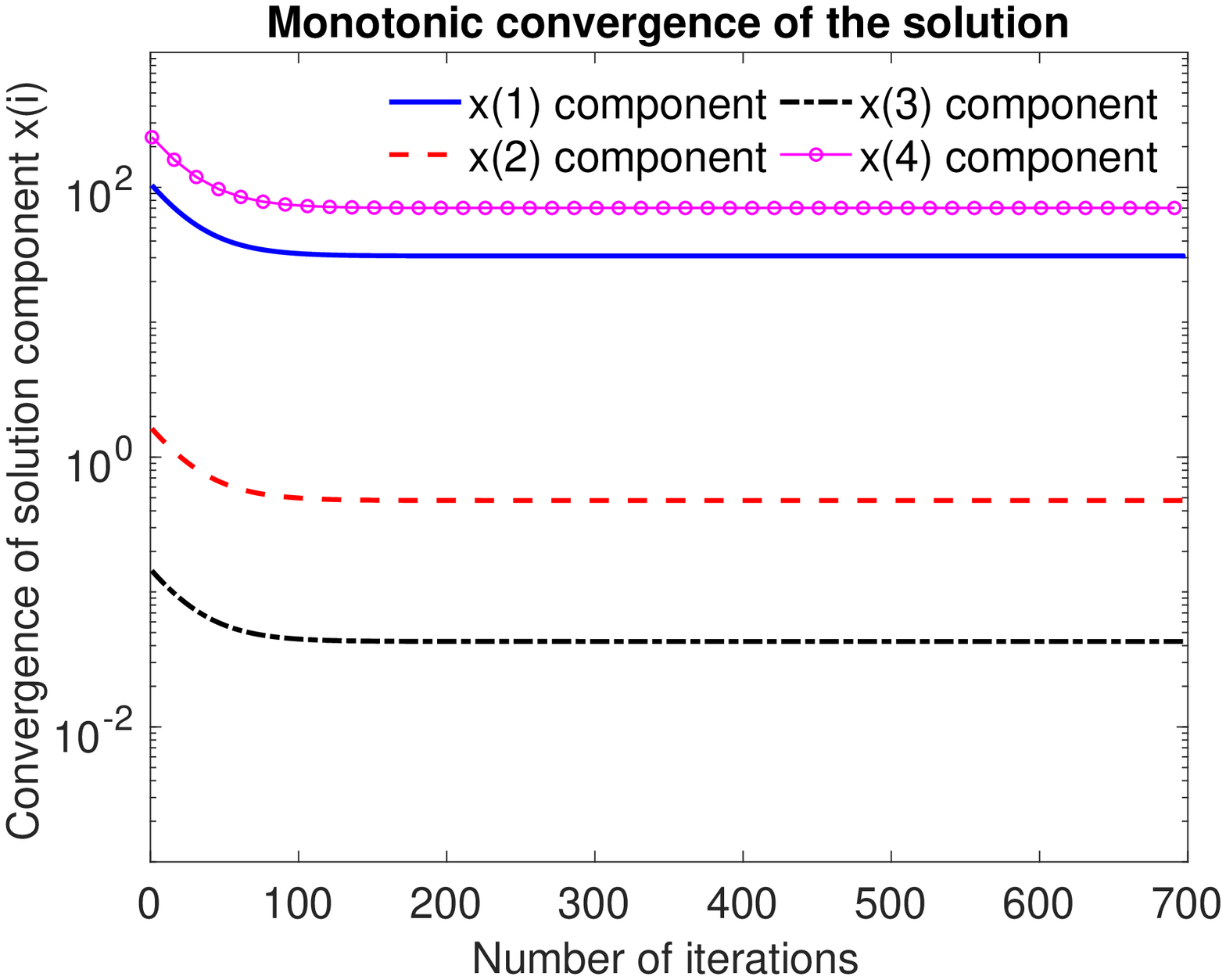}
    \caption{Monotonically increasing and decreasing convergence pattern of the iterative solution to the exact solution from two different initial approximated vector, which are nonnegative.}
    \label{fig:fig2}
\end{figure}
\noindent The monotonic convergence theorem proved in Section \ref{monotone} is computationally established by solving the $4 \times 4$ linear system (\ref{multirhs}) with multiple Right-Hand Side(RHS). The matrix and both the splittings satisfy all the required conditions mentioned in the theorem. Also, the initial approximations $x_0= {[0, 0, 0, 0]}^T$ and $y_0 ={[106.5647, 10, 1, 241.9189]}^T $ satisfy the necessary conditions required by  Theorem \ref{thm4.1}. Only the first column of the RHS matrix is used for the two-stage iteration method to generate Fig.\ref{fig:fig2} corresponding to the iteration numbers. One can observe here, each component of the solution vector converges monotonically. In (a), the convergence is monotonically increasing.  In (b),  it is monotonically decreasing. And one can observe  from the above figure that both are converging to single solution vector ${{\mathcal{A}}^{-1}}\mathcal{B}(:,1)$.\\

\noindent Next, our interest is to understand the computational aspect of the two-stage iterative method using the type-II splittings. Our matrix computations considered the $4 \times 4$ transition matrix (\ref{Anum}) of the pandemic model with standard iteration scheme \eqref{eqn1.2} and two-stage iteration scheme \eqref{tsiter}, and similarly an extended $64 \times 64$ matrix using the block matrix formulation (\ref{Akron}) of (\ref{Anum}). In two-stage Algorithm-\ref{alg1}, we have used SOR type splitting with a relaxation parameter $\omega$. In Table-\ref{table:one}, we have compared the standard iteration scheme with the two-stage standard iteration scheme corresponding to $\omega = 1$ and $\omega = 1.7$. The data listed in table shows that the two-stage iteration scheme for $\omega = 1.7$ is faster than the standard iteration scheme and the two-stage  iteration scheme with $\omega=1$. \\

\noindent When the condition numbers of the matrices become larger,  the two-stage iteration scheme with  $\omega = 1.7$ converges gradually faster than the two-stage iteration scheme for $\omega = 1$. The condition number is higher when the rate at which the recovered individuals are reinfected (or $\phi$ value) in the model is bigger, so we have considered the value of $\phi$ as 0.07, 0.08, 0.09 and 0.10, such that the condition number increase gradually and the iteration numbers also increase. In Table-\ref{table:one}, we have computed condition number only for $4 \times 4$ matrices as there is no significant change in condition number for $64 \times 64$ size matrices when $\phi $ values are same.  Similarly,   we have computed the spectral radius only for $64 \times 64$ size.

\begin{table}[h]
\begin{center}
\caption{Comparison between standard iteration scheme and two-stage iteration scheme}\label{table:one}
\begin{tabular}{ |p{3.6cm}||p{2.2cm}|p{3.5cm}|p{1.5cm}|c|   }
\hline
{\bf $\phi$~-~value} & 	{\bf One stage} & {\bf Two-stage($\omega$=1)} & \multicolumn{2}{|c|}{\bf Two-stage($\omega$=1.7)} \\
\hline\hline
{\bf Matrix size $4 \times 4$ } &\multicolumn{3}{|c|}{\bf No. of iterations } & \bf~~~~~$\kappa(\mathcal{A})~~~~~$ \\
\hline\hline
0.07 	&136 	&68 	&71 	& 27.36\\
\hline\hline
0.08 	&207 	&104 	&83 	& 39.60\\
\hline\hline
0.09 	&380	&190 	&108 	& 69.33\\
\hline\hline
0.10 	&1428 	&714 	&149 	&$2.43\times{10}^{+02}$\\
\hline\hline
{\bf Matrix size $64 \times 64$ } &\multicolumn{3}{|c|}{\bf No. of iterations }& $\rho(T)$ \\
\hline\hline
0.07 	&142 	&71 	&72 	&0.686\\
\hline\hline
0.08 	&218 	&109 	&89	    &0.733\\
\hline\hline
0.09 	&400 	&200 	&116 	&0.778\\
\hline\hline
0.10 	&1496 	&748 	&154 	&0.820\\
\hline
\end{tabular}
\end{center}
\end{table}


\section{Acknowledgements}

The second(NM) and last(DM) authors acknowledge the support provided
by Science and Engineering Research Board, Department of Science and Technology, New Delhi, India, under the grant numbers MTR/2019/001366 and 
MTR/2017/000174, respectively. We would also like to thank the {\bf Government of India} for introducing the {\it work from home initiative} during the COVID-19 pandemic.

\end{document}